\documentclass[11pt]{amsart}
\usepackage{amscd}
\usepackage{amsfonts}
\usepackage{amsmath,amsthm,hyperref}
\usepackage[all]{xy}
\usepackage{amsmath,amsthm,hyperref}
\usepackage{amsmath,amssymb,amsthm,latexsym}
\usepackage{amscd}
\usepackage{color}
\usepackage{enumitem}
\newtheorem{theorem}{Theorem}[section]

\newtheorem{corollary}[theorem]{Corollary}
\newtheorem{lemma}[theorem]{Lemma}

\newtheorem{conjecture}[theorem]{\bf Conjecture}
\numberwithin{equation}{section}
\theoremstyle{remark}
\newtheorem{remark}[theorem]{Remark}

\newcommand{\Ind}{\operatorname{Ind}}
\newcommand{\Span}{\operatorname{Span}}

\newcommand{\Spec}{\operatorname{Spec}}
\newcommand{\dive}{\operatorname{div}}

\def\<{\langle}
\def\>{\rangle}
\def\SS{\mathbb S}
\def\SSS{\mathbb{S}^{n+1}}
\begin{document}
\title[Index of minimal hypersurfaces in $\SSS$ with $\lambda_1<n$]{On the index of minimal hypersurfaces in $\mathbb{S}^{n+1}$ with $\lambda_1<n$}


\author{Hang Chen}
\address[Hang Chen]{School of Mathematics and Statistics, Northwestern Polytechnical University, Xi' an 710129, P. R. China \\ email: chenhang86@nwpu.edu.cn}
\thanks{Chen was supported by Natural Science Foundation of Shaanxi Province Grant (No.~2024JC-YBMS-011) and Shaanxi Fundamental Science Research Project for Mathematics and Physics (Grant No.~22JSQ005)}
\author{Peng Wang}
\address[Peng Wang]{School of Mathematics and Statistics, Key Laboratory of Analytical Mathematics and Applications (Ministry of Education), FJKLAMA, Fujian Normal University, 350117 Fuzhou, P.R. China \\ pengwang@fjnu.edu.cn, netwangpeng@163.com}
\thanks{
Wang was supported by the Projects 12371052 of NSFC}

\date{\today}
 
\begin{abstract}
In this paper, we prove that a closed minimal hypersurface in $\SSS$ with $\lambda_1<n$ has Morse index at least $n+4$, providing a partial answer to a conjecture of Perdomo. As a corollary, we re-obtain a partial proof of the famous Urbano Theorem for minimal tori in $\mathbb{S}^3$: a minimal torus in $\mathbb{S}^3$ has Morse index at least $5$, with equality holding if and only if it is congruent to the Clifford torus.
The proof is based on a comparison theorem between eigenvalues of two elliptic operators, which also provides us simpler new proofs of some known results on index estimates of both minimal and $r$-minimal hypersurfaces in a sphere.
\end{abstract}



\maketitle

\section{Introduction}

The study of minimal hypersurfaces in the unit sphere $\SSS$ play important roles in global differential geometry. The work \cite{Sim68} of Simons on the stability problems of minimal submanifolds made influential progress in several directions including submanifold theory and geometric partial differential geometry. In particular, he showed that any closed minimal hypersurface in $\SSS$ has  Morse index greater than or equal to $1$, with equality holding if and only if it is the totally geodesic great $n$-sphere.

In  \cite{Urb90}, Urbano proved that any non-totally-geodesic, closed, minimal surface in $\mathbb{S}^3$ has index at least $5$ and with equality holding if and only if it is the  Clifford torus (and hence has genus 1). This is the key characterization to pick up the Clifford torus among all minimal surfaces in $\SS^3$ in the proof of the Willmore conjecture   in $\mathbb{S}^3$ by Marques and Neves \cite{MN14a}. As they pointed out in \cite{MN14}, strong evidence for the analogous conjecture in higher codimension would be a generalization of Urbano's result to minimal $2$-tori in $\SSS$. Very recently, Kusner and the second author \cite{KW24} proved an Urbano-type theorem for minimal $T^2$ in $\mathbb{S}^4$.

For higher dimensional cases, El Soufi \cite{ESou93} showed that any non-totally-geodesic, closed, minimal hypersurface in $\SSS$ has index at least $n+3$, and the index of Clifford torus exactly equals to $n+3$, which partially extended Urbano's result \cite{Urb90} for $n=2$. In \cite{Per01}, Perdomo re-obtained this result and proposed naturally the following conjecture.
\begin{conjecture}[cf. \cite{Per01}]\label{conj-perdomo}
Let $x: M^n\rightarrow\SSS$ be a full, closed, minimal hypersurface in $\SSS$, $n\ge 3$. Then $\Ind(M)\geq n+3$, with equality holding if and only if $x(M)$ is congruent to one of the Clifford minimal hypersurfaces $\mathbb{S}^{m}(\sqrt{\frac{m}{n}})\times \mathbb{S}^{n-m}(\sqrt{\frac{n-m}{n}})$, $1\leq m\leq \left[\frac{n}{2}\right]$.
\end{conjecture}
Perdomo \cite{Per01} proved that the conjecture holds if $x$ has antipodal symmetries. And we refer to  \cite{GBD99, dBS09} or Theorem \ref{thm-app1-min} for some proofs of the conjecture under some extra assumptions. So far this conjecture stays open.

In this note, we proved Perdomo conjecture under the assumption
$\lambda_1<n$, which is slightly relevant to the famous Yau's conjecture that every closed embedded minimal hypersurface in $\SSS$ has $\lambda_1=n$. For simplicity here we denote $\lambda_0=0$.

\begin{theorem}\label{thm-main}
    Let $x: M^n\rightarrow\SSS  $ be a full minimal hypersurface. Let $E^{\Delta}_{\lambda_j}$ be the space of eigenfunctions of the Laplacian on $M$ with respect to the eigenvalue $\lambda_j$.
    Then
    \begin{equation*}
        \Ind(M)\geq \sum_{0\leq \lambda_j\leq n}\dim E^{\Delta}_{\lambda_j}.
    \end{equation*}
\end{theorem}

\begin{corollary}
    Let $x: M^n\rightarrow\SSS  $ be a full minimal hypersurface with either $\lambda_1<n$ or $\dim E^{\Delta}_{\lambda_j=n}\geq n+3$. Then $\Ind(M)\geq n+4$.
\end{corollary}

Since the Clifford torus is the unique minimal torus in $\mathbb{S}^3$ with $\lambda_1=2$ \cite{MR86} (see also \cite{ESI00} for a classification of all minimal 2-tori in $\mathbb{S}^n$ with $\lambda_1=2$), we have the following new proof of the (partial) Urbano Theorem for minimal tori in $\mathbb{S}^3$.
\begin{corollary}[cf. \cite{Urb90}]
    Let $x: T^2\rightarrow\mathbb{S}^3$ be a minimal torus. Then $\Ind(M)\geq 5$ with equality if and only if $x(M)$ is congruent to the Clifford torus.
\end{corollary}

\begin{remark}
    \
\begin{enumerate}
    \item Note that Urbano's original theorem holds for all minimal surfaces in $\SS^3$ of genus greater or equal to $1$. Here we can only prove it for the case of torus.

    \item Recall that by Kapouleas and Wiygul \cite{KW20}, the Lawson minimal surface $\xi_{g,1}$ in $\mathbb{S}^3$ has Morse index $2g+3$, where $g$ is the genus of $\xi_{g,1}$ and by the work of Choe-Soret \cite{CS09} we have $\lambda_1(\xi_{g,1})=2$.
    So the above corollary and our proofs can not apply for minimal surfaces of higher genus in $\mathbb{S}^3$. So it is only a partial proof of Urbano's Theorem.

    \item And we also note that Savo proved an important estimate via first Betti numbers \cite{Sav10}, which will derive  Urbano Theorem for surfaces of genus $\geq 5$. We also refer to \cite{AB22, ACS18} and references therein for some recent important progress on this direction.

    \item
    We also refer to \cite{Bre13, Eji83, ESI00, Kar21,TY13,TXY14,WW23} for some recent progress on the index estimate and Yau's conjecture for minimal surfaces or minimal hypersurfaces.
    \end{enumerate}
\end{remark}

Our proof  of Theorem \ref{thm-main} comes from a simple comparison theorem between  eigenvalues of two elliptic operators, which we refer to  Theorem \ref{thm-operator-comparison} of Section 2 for full details.  The Comparison Theorem \ref{thm-operator-comparison} can be also applied for index estimates of $r$-minimal hypersurfaces in $\SSS$, which provides a unified, new proof of some known results on this direction.

The paper is organized as follows. In Section 2 we will derive the comparison theorem between eigenvalues of two elliptic operator. Section 3 is devoted to the proof of Theorem \ref{thm-main} and we end the paper with a new proof of index estimates of $r$-minimal hypersurfaces in $\SSS$.

\section{Comparison of eigenvalues of two self-adjoint ellptic operators}
Let $M$ be a closed Riemannian manifold and let $L$ be a self-adjoint, elliptic operator acting on $C^\infty(M)$.
Given a positive smooth function $p$ on $M$.
We consider the weighted eigenvalue problem
\begin{equation}\label{eq-WEP}
    Lu=-\lambda pu.
\end{equation}
If there exists a real numbrer $\lambda$ such that \eqref{eq-WEP} holds for some nonzero function $u$,
then we call $\lambda$ is a $p$-weighted eigenvalue of $L$ and
$u$ is called a $p$-weighted eigenfunction corresponding to $\lambda$.
It is well known that (cf. \cite[Chapter III]{Ban80}) the $p$-weighted eigenvalues of $L$ form a discrete sequence $\{\lambda_j\}_{j=0}^{+\infty}$ satisfying $\lambda_1\le \lambda_2\leq \cdots\rightarrow+\infty$,
and we can find an orthonormal set (in $p$-weighted $L^2$ sense) of eigenfunctions $\{u_j\}_{j=0}^{+\infty}$ such that $u_j$ is corresponding to $\lambda_j$, precisely,
\begin{equation}
    \int_M u_i u_j p\,dv=0 \mbox{ for $i\neq j$}.
\end{equation}

Now set
\begin{gather}
    \Spec(L, p):=\{\lambda_j\mid 1\leq j\leq +\infty\};\\
    N^{L,p}_{<a}:=\#\{\lambda\in \Spec(L,p)\mid \lambda<a\}, \quad N^{L,p}_{\leq a}:=\#\{\lambda\in \Spec(L,p)\mid \lambda\leq a\},\label{eq-def-N}\\
    N^{L, p}_{=a}:=\#\{\lambda\in \Spec(L,p)\mid \lambda=a\}=N^{L,p}_{\leq a}-N^{L,p}_{< a},
\end{gather}
with multiplicity.

When $p\equiv 1$, \eqref{eq-WEP} becomes the usual eigenvalue problems; in this case, we omit $p$ in above notations and just write $\Spec(L), N^{L}_{<a}, N^{L}_{\le a}, N^{L}_{=a}$ for simplicity.

Now we state the comparison theorem as follows.
\begin{theorem}\label{thm-operator-comparison}
    Let $M$ be a closed Riemannian manifold.
Let  $L$ and $\hat{L}$  be two self-adjoint, elliptic operators on $M$, satisfying $\hat{L}=L+q$ for some smooth function $q\in C^\infty(M)$.
    Then for any real number $a\in \mathbb R$, and any positive function $p\in C^\infty(M)$, we have
\begin{enumerate}
    \item if $q/p$ is not constant, then
    \begin{equation}\label{eq-N-minf}
        N^{\hat{L},p}_{<a-a_0}\ge N^{L,p}_{\le a},
    \end{equation}
where $a_0:=\inf_M\frac{q}{p}$;
\item    if $q/p$ is constant, then
    \begin{equation}\label{eq-N-f-const}
        N^{\hat{L},p}_{\le a-q/p}=N^{L,p}_{\le  a} \mbox{ and } N^{\hat{L},p}_{<a-q/p}=N^{L,p}_{< a}.
    \end{equation}
\end{enumerate}

\end{theorem}

\begin{proof}
    For $\lambda_i\in \Spec(L,p)$, suppose $k\in \mathbb{Z}^{+}$ such that $\lambda_k\le a$ and $\lambda_{k+1}>a$.
    Denote
    \begin{equation}
        V=\Span\{u_1,\dots,u_k\}.
    \end{equation}
    Considering the Rayleigh quotient for the weighted eigenvalue problem \eqref{eq-WEP}
    \begin{equation}
        Q^{L,p}(u):=-\frac{\int_M u Lu \,dv}{\int_M u^2 p\, dv}, \quad u \in C^\infty(M)\backslash\{0\},
    \end{equation}
    we have
    \begin{equation}
        Q^{L,p}(u)\le a, \mbox{ for any } u\in  V\backslash\{0\}.
    \end{equation}
    Hence, for any $u\in V\backslash\{0\}$, the Rayleigh quotient for $\hat{L}$ with weight $p$ satisfies
    \begin{equation}\label{eq-2.8}
        \begin{aligned}
            Q^{\hat{L},p}(u) &= -\frac{\int_M u \hat{L}u \,dv}{\int_M u^2p\,dv}\\
            &=-\frac{\int_M u (L+q)u \,dv}{\int_M u^2 p\,dv}\\
            & =Q^{L,p}(u)-\frac{\int_M (q/p) u^2 p\,dv }{\int_M u^2p\,dv}\\
            & \le a -a_0,
        \end{aligned}
    \end{equation}
where $a_0=\inf_M\frac{q}{p}$.

    (1) If $q/p$ is not constant, then the last inequality in \eqref{eq-2.8} should become strict since
    \begin{equation*}
        \int_M (\frac{q}{p}-\inf \frac{q}{p}) u^2 p\,dv>0.
    \end{equation*}
    So we obtain
    \begin{equation}
        Q^{\hat{L},p}(u) <a -\inf \frac{q}{p}=a-a_0, \mbox{ for any } u\in  V\backslash\{0\},
    \end{equation}
    which implies \eqref{eq-N-minf}.

    (2) If $q/p$ is constant, then $q/p=\inf(q/p)=\sup (q/p)$ and it follows from \eqref{eq-2.8} that
    \begin{equation}
        N^{\hat{L},p}_{\le a- q/p}\ge N^{L,p}_{\le a}.
    \end{equation}
    On the other hand, \eqref{eq-2.8} for $L=\hat{L}-q$ implies
    \begin{equation}
        N^{L,p}_{\le a}=N^{L,p}_{\le (a-q/p)-\inf (-q/p)} \ge N^{\hat{L},p}_{\le a-q/p},
    \end{equation}
    hence, $N^{\hat{L},p}_{\le a- q/p}=  N^{L,p}_{\le a}$.
    Similarly, we obtain $N^{\hat{L},p}_{<a- q/p}=  N^{L,p}_{<a}$.
\end{proof}

\section{Index estimate of minimal hypersurfaces in a sphere}
In this section, we will use Theorem \ref{thm-operator-comparison} to estimate the index of minimal hypersurfaces.
Firstly, we recall some basic conceptions and facts about minimal hypersurfaces, which can be found in many related literatures, e.g., \cite{Sim68, CM11}.

Let $x: M^n\rightarrow \SSS$ be a minimal immersion of a closed Riemannian manifold $M$ into $\SSS$.
Denote by $\nu $ and $S$ the unit normal vector field and the squared norm of the second fundamental form of $M$, respectively.
It is well known that minimal hypersurfaces are critical points of the area functional,
precisely,
consider any normal variation $x_t$ of $x$ such that $x_0=x, \frac{\partial}{\partial t}\big|_{t=0}x_t=u \nu$
and let $\mathcal{A}(t)$ denote the area of $x_t$,
then $\mathcal{A}'(0)=0$.
Here $\nu$ denotes the unit normal vector field on $M$ and $u\in C^\infty(M)$ is a smooth function.
Furthermore, the {\em second variation formula} of
a minimal immersion $x: M^n\rightarrow \SSS$
is given by
\begin{equation}\label{index-form}
    \mathcal{A}''(0)=-\int_M uJu\,dv=:Q(u,u),
\end{equation}
where
\begin{equation}\label{eq-L}
    J:=\Delta+n+S
\end{equation}
is usually called the {\em (volume) Jacobi operator}.
Clearly, $J$ is a strongly elliptic, self-adjoint operator acting on $  C^\infty(M)$.

The {\em Morse index} of a minimal immersion $x: M^n\rightarrow \SSS$, denoted by $\Ind(M)$, is defined to be the index of the quadratic form $Q$ associated \eqref{index-form} with the second variation of $x$.
In other words, $\Ind(M)$ is the maximal dimension of a subspace of  $ C^\infty(M)$ on which $Q$ is negative definite, that is, the sum of the dimensions of the
eigenspaces of $J$ belonging to its negative eigenvalues.

\begin{proof}[Proof of Theorem \ref{thm-main}]
Set
\begin{equation*}
    L=\Delta,\ \hat{L}=\Delta+q,\ q=n+S,\ p\equiv 1 \mbox{ and } a=\inf q
\end{equation*}
in Theorem \ref{thm-operator-comparison}.  Since $a_0=\inf q\ge n$, we obtain by (1) of  Theorem \ref{thm-operator-comparison} that
$N^{\hat{L},p}_{<a-a_0}\ge N^{L,p}_{\le a}$, which yields Theorem  \ref{thm-main} directly.
\end{proof}

When $n\ge 3$, it has been proven that Conjecture \ref{conj-perdomo} holds true under some extra assumptions in \cite{GBD99, dBS09} respectively.
Here we can derive a new, unified proof of their results by use of Theorem \ref{thm-operator-comparison}.

\begin{theorem}[\cite{GBD99, dBS09}]\label{thm-app1-min}
    Let $x:M^n\rightarrow\SSS$ be a full, closed minimal hypersurface.
    Suppose one of the following conditions holds:
\begin{enumerate}[label=(\alph*)]
    \item  $S$ is constant;
    \item  $S\ge n$.
\end{enumerate}
Then Conjecture \ref{conj-perdomo} holds true. Moreover, if $x(M)$ is not congruent to one of the Clifford minimal hypersurfaces, then $\Ind(M)\ge 2n+5$.
\end{theorem}

To begin with, we first recall the following well-known fact.
\begin{lemma}[cf. \cite{Sim68, Urb90}]\label{lem-Lap}
    For any minimal immersion $x: M^n\to \SSS$, let $x^i$ and $\nu^i$ be the coordinate functions of the position vector $x=(x^1,\cdots, x^{n+2})$ and the unit normal $\nu=(\nu^1,\cdots, \nu^{n+2})$ respectively. Then for each $i$, $1\le i\le n+2$, we have
\begin{equation}
    \begin{cases}
    \Delta x^i=-n x^i,\\
    \Delta \nu^i=-S \nu^i.
    \end{cases}
\end{equation}
\end{lemma}

\begin{proof}[Proof of Theorem \ref{thm-app1-min}]
    We still take $L=\Delta$, $q=n+S$ and $p\equiv 1$ in Theorem \ref{thm-operator-comparison}.
    Since $x$ is full, we have
    \begin{equation*}
        N^{\Delta}_{\le n}\ge n+3.
    \end{equation*}

    (a) If $S$ is constant, then $S>0$ since $x$ is full.
    By taking $a=\inf q=n+S$ in Theorem \ref{thm-operator-comparison}, we have
    \begin{equation}
        N^J_{<0}= N^{\Delta}_{< n+S}\ge N^{\Delta}_{\le n}\ge n+3.
    \end{equation}

    If $S\le n$, then by the famous rigidity theorem (cf. \cite{Law69,CdCK70}), we have $x(M)$ is congruent to one of the Clifford minimal hypersurfaces and it has index  $n+3$.

    If $S>n$, then by Lemma \ref{lem-Lap}, we know that $n$ and $S$ are two different eigenvalues of $\Delta$ and both of them have multiplicities at least $n+2$. Since the first eigenvalue of $\Delta$ is $0$, we obtain
    \begin{equation}
        N^J_{<0}= N^{\Delta}_{< n+S}\ge 1+(n+2)+(n+2)=2n+5.
    \end{equation}

    (b) If $S\ge n$, without loss of generality, we suppose $S$ is not constant. Otherwise, it reduces to the above arguments.
    By taking $a=n$ in Theorem \ref{thm-operator-comparison}, and noting that  $-n\ge -\inf S$, we obtain
    \begin{equation}
        N^J_{<-n}\ge N^J_{<-\inf S}=N^J_{<n-\inf q}\ge  N^{\Delta}_{\le n}\ge n+3.
    \end{equation}
    On the other hand, Lemma \ref{lem-Lap} implies that $-n$ is the eigenvalues of $J$ with multiplicity at least $n+2$,
    \begin{equation}
        N^J_{<0}\ge N^J_{<-n}+N^J_{=-n}\ge (n+3)+(n+2)=2n+5.
    \end{equation}
\end{proof}

\section{Index estimate of $r$-minimal hypersurfaces in a sphere}
From the view of the variations, the conception and many related results of minimal hypersurfaces in a sphere have been generalized to $r$-mimimal hypersurfaces.
Here we just give a quick review; the readers can refer \cite{Rei73, Ros93, BC97, CL07} and the references therein for more details.

A hypersurface in $\SSS$ is called \emph{$r$-minimal} if and only if its $(r+1)$-th mean curvature $H_{r+1}:=S_{r+1}\big/\binom{n}{r+1}$ vanishes, where $S_r$ is the $r$-th elementary symmetric function of the principal curvatures $k_1,\dots, k_n$, i.e.,
\begin{equation*}
    S_0=1,\ S_1=k_1+\dots+k_n,\ \cdots,\ S_n=k_1\cdots k_n.
\end{equation*}
The $r$-minimal hypersurfaces are the critical points of the functionals
\begin{equation}
    \mathcal{A}_r=\int_M F_r(S_1,\dots,S_r)\,dv,
\end{equation}
where $F_r$ are defined inductively by
\begin{equation}
    \begin{cases}
    F_0=1, F_1=S_1,\\
    F_r=S_r+\frac{n-r+1}{r-1}F_{r-2}, & \mbox{for $2\le r\le n-1$}.
    \end{cases}
\end{equation}

Let $T_r$ denote the $r$-th Newton transformation arising from the shape operator $A$, i.e.,
\begin{equation*}
    T_0=I, T_r=S_rI-AT_{r-1}.
\end{equation*}
We define the operator $L_r$ by
\begin{equation}
    L_r u:=\dive (T_r \nabla u),
\end{equation}
then the second variation formula (for $r$-minimal hypersurfaces) is given by
\begin{equation}\label{index-form-r}
    \mathcal{A}_r''(0)=-(r+1)\int_M u J_r u\,dv=:Q_r(u),
\end{equation}
where
\begin{equation}
    J_r=L_r + (n-r) S_r -(r+2)S_{r+2}.
\end{equation}
Similarly, we can also define the index of $r$-stability, denoted by $\Ind^r(M)$,
to be the index of the quadratic form $Q$ associated \eqref{index-form-r} with the second variation of $x$. In other words, $\Ind^r(M)$ is equal to the maximal dimension of a subspace of  $ C^\infty(M)$ on which $Q_r$ is negative definite, that is, the sum of the dimensions of the eigenspaces of $J_r$ belonging to its negative eigenvalues.

Theorem \ref{thm-app1-min} has been generalized to $r$-minimal hypersurfaces in $\SSS$ by Barros and Sousa in \cite{BS10}.
\begin{theorem}[{\cite[Theorems 1]{BS10}}]\label{thm-app-r-min}
    Given $r\in \{1,..., n-1\}$, let $x: M^n\rightarrow\SSS$ be an $r$-minimal hypersurface with $S_{r+2}<0$.
    Then $\Ind^r(M)\ge n+3$.
\end{theorem}

\begin{theorem}[{\cite[Theorems 3 and 4]{BS10}}]\label{thm-app-r-min2}
    Given $r\in \{1,..., n-1\}$, let $x: M^n\rightarrow\SSS$ be an $r$-minimal hypersurface.
    Suppose one of the following conditions holds:
\begin{enumerate}[label=(\alph*)]
    \item $S_{r+2}/S_r$ is a nonzero constant;
    \item  $-(r+2)S_{r+2}\ge (n-r)S_r$ and $S_r>0$.
\end{enumerate}
    Then
    \begin{enumerate}
    \item $\Ind^r(M)\ge n+3$;
    \item  $\Ind^r(M)$ is equal to $n+3$ if and only if
    $x(M)$ is one of (generalized) Clifford torus;
    \item  $\Ind^r(M)\ge 2n+5$ when $\Ind^r(M)>n+3$.
    \end{enumerate}
\end{theorem}

\begin{remark}
    Since $S_0=1>0$ and $2S_2=-S$, Theorem \ref{thm-app-r-min2} recovers Theorem \ref{thm-app1-min} when $r=0$.
\end{remark}

To provide new proofs of Theorems \ref{thm-app-r-min} and \ref{thm-app-r-min2} by using Theorem \ref{thm-operator-comparison}, let us first recall the following generalization of Lemma \ref{lem-Lap}.
\begin{lemma}[cf. \cite{Rei73, Ros93}]\label{lem-Lap-2}
    For any $r$-minimal immersion $x: M^n\to \SSS$, let $x^i$ and $\nu^i$ be the coordinate functions of  $x=(x^1,\cdots, x^{n+2})$ and $\nu=(\nu^1,\cdots, \nu^{n+2})$ respectively. Then for each $i$, $1\le i\le n+2$, we have
\begin{equation}
    \begin{cases}
        L_r x^i=-(n-r)S_r x^i, \\
        L_r \nu^i=(r+2)S_{r+2}\nu^i.
    \end{cases}
\end{equation}
\end{lemma}

\begin{proof}[Proof of Theorem \ref{thm-app-r-min}]
    First of all, we point out that the conditions $r$-minimality (i.e. $S_{r+1}\equiv 0$) and $S_{r+2}<0$ imply that $S_r>0$ and $L_r$ is elliptic;
    see \cite{HL99} and the original proof of \cite[Theorems 1]{BS10} for more details.

    Take
    \begin{equation*}
        L=L_r+(n-r)S_r,\ q=-(r+2)S_{r+2},\ p\equiv 1 \mbox{ and } a=\inf q
    \end{equation*}
    in Theorem \ref{thm-operator-comparison}. Since $S_{r+2}<0$, we have $a>0$, and then
    \begin{equation}
        N^{J_r}_{ < 0}\ge N^L_{ < a} \ge N^L_{\le 0},
    \end{equation}
    whether $S_{r+2}$ is constant or not.

    Now take $q=(n-r)S_r$ and $a=\inf q$ in Theorem \ref{thm-operator-comparison} and treat $\{L, L_r\}$ as $\{\hat{L}, L\}$ respectively,
    we have
    \begin{equation}
        N^L_{ < 0}\ge N^{L_r}_{ < (n - r)\inf S_r}\ge N^{L_r}_{\le 0} \ge 1,
    \end{equation}
    where we used $S_r>0$ and the fact that $0$ is an eigenvalue of $L_r$.

    By Lemma \ref{lem-Lap-2}, $Lx^i=0$, which means $0$ is an eigenvalue of $L$ with multiplicity at least $n+2$.
    Hence,
    \begin{equation}
        N^{J_r}_{ < 0}\ge N^L_{\le 0} = N^L_{ <0} +N^L_{ = 0}\ge n + 3.
    \end{equation}
\end{proof}

\begin{remark}
    \
    \begin{enumerate}
    \item
    In the above proof, we used Theorem \ref{thm-operator-comparison} to prove $N^L_{\le 0}\ge n+3$.  This result can be obtained by using the fact that the first eigenvalue of $L$ is simple
    (so it must be less than $0$).
    \item
    When $r=0$, $S_1=0$ means $x(M)$ is minimal.
    $2S_2=-S<0$ also holds automatically if $x(M)$ is non-geodesic.
\end{enumerate}
\end{remark}

\begin{proof}[Proof of Theorem \ref{thm-app-r-min2}]
The conclusion (1) follows from Theorem \ref{thm-app-r-min}.
Indeed, under the condition (a), again by Hounie-Leite's result \cite{HL99} (see the original proof of \cite[Theorems 4]{BS10} for details),
we conclude that $S_r>0$ and $L_r$ is elliptic;
furthermore, $0=H_{r+1}^2\ge H_{r}H_{r+2}$ (cf. \cite[Proposition 1]{Cam06}) implies $S_{r+2}<0$.
Under the condition (b),
we have $S_{r+2}<0$ and then $L_r$ is elliptic.

Next, we prove (2) and (3) based on (1).
It is essentially the same as the proof of Theorem \ref{thm-app1-min}.
However, different from the case of minimal hypersurfaces, $S_r$ may not be constant in this situation.
Fortunately, note that
    \begin{equation}
        -\frac{\int_M uLu\, dv}{\int_M u^2\, dv}<0 \iff -\frac{\int_M uLu\, dv}{\int_M u^2 p\, dv}<0,
    \end{equation}
    therefore, there are no differences between by using the weighted Rayleigh quotient and by using the usual Rayleigh quotient when estimating the index.
    This indicates us to take
    \begin{equation*}
        p=S_r>0,~L=L_r,~ \hat{L}=J_r=L+q
    \end{equation*}
    with $q=(n-r)S_r-(r+2)S_{r+2}$ in Theorem \ref{thm-operator-comparison}.

    Keep in mind $L_r x=-(n-r)S_r x, L_r \nu=(r+2)S_{r+2}\nu$ (Lemma \ref{lem-Lap-2}).

    \textbf{Case (a).} When $S_{r+2}/S_r$ is constant, $q/p$ is also constant and note that
    \begin{equation*}
    q/p=(n-r)-(r+2)S_{r+2}/S_r>(n-r).
    \end{equation*} Taking $a=q/p$, we have
    \begin{equation}
        N^{J_r,p}_{<0}=N^{L_r,p}_{<p/q}\ge  N^{L_r}_{\le n-r}\ge n+3.
    \end{equation}

    If $-(r+2)S_{r+2}/S_r\le n-r$, then $M$ is a generalized Clifford torus satisfying $S_{r+1}=0$ (by Caminha's rigidity theorem \cite[Theorem 4.5]{Cam06a}) and its $r$-index is just $n+3$ (\cite[Theorem 2]{BS10}). So the conclusion (2) is proven.

    If $-(r+2)S_{r+2}/S_r> n-r$, then  $n-r$ and $-(r+2)S_{r+2}/S_r$ are the different ($S_r$-weighted) eigenvalues of $L_r$ and both of them have multiplicities at least $n+2$. So we have
    \begin{equation}
        N^{J_r,p}_{<0}=N^{L_r,p}_{<p/q}\ge 1+(n+2)+(n+2)=2n+5.
    \end{equation}
    This is the conclusion (3).

    \textbf{Case (b).} When $-(r+2)S_{r+2}\ge (n-r)S_r$ and $S_r>0$.
    Suppose $S_{r+2}/S_r$ is not constant without loss of generality; otherwise, it reduces
    to Case (a).

    By taking $a=n-r$ and noting that
    \begin{equation*}
        -(n-r)\ge -(r+2)\inf (-S_{r+2}/S_r),
    \end{equation*}
    we obtain
    \begin{equation}
        N^{J_r,p}_{ < -(n-r)}\ge N^{J_r,p}_{ < -(r+2)\inf (-S_{r+2}/S_r)}=N^{J_r,p}_{ < (n-r)-\inf (q/p)}\ge N^{L_r,p}_{ \le n-r}\ge n+3.
    \end{equation}
    Finally, we conclude
    \begin{equation}
        N^{J_r,p}_{ \le -(n-r)}=N^{J_r,p}_{<-(n-r)}+N^{J_r,p}_{=-(n-r)}\ge 2n+5,
    \end{equation}
    since $N^{J_r,p}_{=-(n-r)}=n+2$.
    So the conclusion (3) is proven.
\end{proof}

\textbf{Acknowledgements}  The authors are thankful to Prof. Ying L\"{u} and Prof. Zhenxiao Xie for valuable discussions.


\end{document}